\documentclass[11pt]{amsart}

\usepackage{graphicx}
\usepackage{amsmath,amssymb,amsthm}
\usepackage[english]{babel}
\usepackage[autostyle, english = american]{csquotes}
\MakeOuterQuote{"}
\usepackage[letterpaper,top=2cm,bottom=2cm,left=3cm,right=3cm,marginparwidth=1.75cm]{geometry}
\usepackage{hyperref}
\hypersetup{colorlinks=true,linkcolor=black,citecolor=black,urlcolor=black}

\theoremstyle{definition}
\newtheorem{definition}{Definition}
\theoremstyle{plain}
\newtheorem{theorem}{Theorem}
\newtheorem{lemma}{Lemma}
\newtheorem{conjecture}{Conjecture}

\title{A counterexample to a mixing-time conjecture for repeated averages on graphs}
\author{Nikash Gupta}

\date{September 2026}

\begin{document}
\maketitle

\section*{Abstract}
The repeated averages process is a stochastic averaging process on graphs whose mixing time is known for several structured families, but no general sharp expression is known for all connected graphs. It was conjectured that the $L^2\!\to\!L^1$ mixing time is of order $|E|\log n/\lambda_2$. We disprove this conjecture using the graph $G_n$ obtained by attaching one leaf to the complete graph $K_n$. We prove that $\lambda_2(G_n)=1$ and that
\[
t_{\varepsilon,2\to1}(G_n)=\Theta_{\varepsilon}(\gamma(G_n))=\Theta_{\varepsilon}(n^2),
\]
with no additional $\log n$ factor. The example has a strongly localized Fiedler eigenvector: most of its squared $L^2$ mass lies on the leaf, while the balancing mass is spread across the clique.

\section{Introduction}

\subsection{Background and motivation}
Suppose that there are $n$ people on a social network to discuss politics. We would expect that when two people interact on this network, their political views become closer to each other after they share their opinions. Then we would expect that after sufficiently many interactions, the range of political opinions among all $n$ people should have decreased significantly. In other words, after some time, a consensus should form in the group. The study of how opinions change under these circumstances is part of a broader field known as opinion dynamics. A widely used mathematical model for opinion dynamics is the averaging process. See, e.g., \cite{hegselmann2005opinion} by Hegselmann and Krause.

\medskip
The averaging process is defined as follows. Assume we have a connected graph $G = (V, E)$ on $n$ vertices and a starting vector $v(0) = (v_1(0), \dots, v_n(0))$. At time $t$, select an edge $e = (i, j)$ from $E$ uniformly at random. Then, to determine $v(t + 1)$ from $v(t)$, preserve all entries except for the $i$th and $j$th entries. Update the $i$th and $j$th entries with the following averaging rule:
\begin{center}
\[
v_{i}(t + 1) = v_j(t + 1) = \frac{v_i(t) + v_j(t)}{2}.
\]
\end{center}
Observe that the sum of the entries in $v(t)$ does not depend on $t$. In the context of the social network, the interpretation of this averaging rule is that whenever two people discuss, they concede an equal part of their opinions. Furthermore, the graph $G$ tells us which pairs from the $n$ individuals may interact at each step.

\medskip
In 2012, Aldous and Lanoue \cite{aldous2012lecture} studied theoretical aspects of the averaging process and proved universal convergence guarantees; they also investigated connections to Markov chains. More recently, Chatterjee, Diaconis, Sly, and Zhang \cite{chatterjee2022phase} analyzed the complete graph and discovered a cutoff phenomenon, adapting the notion of cutoff from finite Markov chains (Diaconis, 1996) \cite{diaconis1996cutoff}.

\medskip
Inspired by these advances, Movassagh, Szegedy, and Wang \cite{movassagh2022repeatedaveragesgraphs} studied repeated averages on broad graph families. They studied $L^1$ and $L^2$ metrics, proved general contraction bounds, and formulated two conjectures about the order of mixing times. These conjectures are as follows.

\begin{conjecture}
Let $G$ be any connected graph with $n$ nodes,
\[
t_{\varepsilon,1}(G)
=
\Theta_{\varepsilon}\bigl(\max\{\gamma(G),n\log n\}\bigr).
\]
\end{conjecture}

\begin{conjecture}
Let $G$ be any connected graph with $n$ nodes,
\[
t_{\varepsilon,2\to1}(G)
=
\Theta_{\varepsilon}\bigl(\gamma(G)\log n\bigr).
\]
\end{conjecture}

Conjecture 1 has since been refuted for the hypercube by Pietro Caputo, Matteo Quattropani, and Federico Sau \cite{Caputo_2023}. This paper focuses on Conjecture 2 and presents a simple counterexample: the complete graph with a single attached leaf.
The key phenomenon is \emph{eigenvector localization}: when a Fiedler vector (an eigenvector corresponding to the second-smallest eigenvalue of the Laplacian) concentrates most of its mass on a small part of the graph, the extra $\log n$ factor in the $L^2\!\to\!L^1$ mixing time can disappear.

\subsection{Main finding}
Define the clique $K_n$ to be the graph on $n$ vertices where every pair of vertices is connected by an edge. For each $n$, define $G_n$ to be the graph $K_n$ with an additional vertex called the \emph{leaf} which is connected by a single edge to one vertex in the clique (a \emph{pendant edge}). For the family of graphs $\{G_n\}$, we prove the following main theorem, which serves as a counterexample to Conjecture 2.

\begin{theorem}[Main theorem]\label{thm:main}
For every fixed $\varepsilon\in(0,1)$, as $n\to\infty$,
\[
t_{\varepsilon,2\to1}(G_n)
=
\Theta_{\varepsilon}\bigl(\gamma(G_n)\bigr).
\]
\end{theorem}

This contradicts the \cite{movassagh2022repeatedaveragesgraphs} prediction
\[
t_{\varepsilon,2\to1}(G)
=
\Theta_{\varepsilon}\bigl(\gamma(G)\log n\bigr).
\]
Our proof uses a two-phase argument: a first phase that ``balances'' the mass between the leaf and any vertex within the clique, and a second phase that further mixes the mass within the clique. The first phase has length $\Theta_{\varepsilon}(\gamma(G_n))$, while the second phase has length $\Theta_{\varepsilon}(n\log n)$.
By a spectral analysis of the Laplacian matrix of $G_n$, we will show that altogether, the $L^2\to L^1$ mixing takes $\Theta_{\varepsilon}(\gamma(G_n))$ time to complete, which contradicts Conjecture 2 from \cite{movassagh2022repeatedaveragesgraphs}.

\subsection{Connection to localization}
\cite{movassagh2022repeatedaveragesgraphs} also proved that if a unit Fiedler eigenvector is $\delta$-delocalized, meaning that its $L^1$ norm is at least $\delta\sqrt{n}$ times its $L^2$ norm, then, for fixed $\delta,\varepsilon>0$,
\[
t_{\varepsilon,2\to1}(G)
=
\Theta_{\delta,\varepsilon}\bigl(\gamma(G)\log n\bigr).
\]
Let $A$ denote the clique vertex adjacent to the leaf $\ell$ in $G_n$. A Fiedler eigenvector $f$ for $G_n$ is given by
\[
f_{\ell}=-(n-1),\qquad
f_A=0,\qquad
f_v=1
\quad\text{for every other clique vertex }v.
\]
Thus, the vector is not supported only on the leaf and its neighbor. Rather, it is strongly localized: the leaf carries a fraction $(n-1)/n$ of its squared $L^2$ mass, while the positive balancing mass is spread thinly over the other clique vertices. Moreover,
\[
\frac{\|f\|_1}{\sqrt{|V|}\,\|f\|_2}
=
\frac{2\sqrt{n-1}}{\sqrt{n(n+1)}}
=
\Theta(n^{-1/2}).
\]
This strong localization explains why the extra $\log n$ factor does not appear. Our proof suggests that $\gamma(G)$ is not the only factor that governs $L^2\to L^1$ mixing time and that Fiedler-eigenvector localization must also be considered.

\section{Preliminaries}
Consider the averaging process on a connected graph $G=(V,E)$ where the vector $v(t)$ represents the state of the averaging process at time $t$. We define the $L^p\to L^q$ mixing time as in \cite{movassagh2022repeatedaveragesgraphs}.

\begin{definition}[\cite{movassagh2022repeatedaveragesgraphs}, Def.~1]
\[
t_{\varepsilon,p\to q}(G)
\equiv
\min\Bigl\{
t\in\mathbb{N}
\ \Bigm|\
\forall\,v(0)\text{ with }\|v(0)\|_p=1:
\bigl(\mathbb{E}[\|v(t)-\bar v\|_q^q]\bigr)^{1/q}
\leq\varepsilon
\Bigr\}.
\]
\end{definition}

\noindent where
\[
\bar v
=
\frac{1}{|V|}
\left(\sum_i v_i\right)\mathbf{1}
\]
is the average vector. To put it in words, $t_{\varepsilon,p\to q}(G)$ is the smallest time $t$ such that for any starting vector on the $L^p$ sphere, the expected $L^q$ distance to the convergence vector $\bar v$ at time $t$ is at most $\varepsilon$. Next, we define $\gamma(G)$ as in \cite{movassagh2022repeatedaveragesgraphs}.

\begin{definition}[\cite{movassagh2022repeatedaveragesgraphs}, Def.~2]
\[
\gamma(G)=\frac{|E|}{\lambda_2(G)}.
\]
\end{definition}

Here $\lambda_2(G)$, also called the \emph{spectral gap} of $G$, is the second-smallest eigenvalue of the Laplacian matrix $L$ of $G$. To be precise, we also recall the $\Theta_{\varepsilon}$ notation used in \cite{movassagh2022repeatedaveragesgraphs}.

\begin{definition}[$\Theta_{\varepsilon}$ notation]
We say that
\[
f(x,\varepsilon)=\Theta_{\varepsilon}(g(x,\varepsilon))
\]
if and only if there exist constants $c_{\varepsilon},C_{\varepsilon}>0$, depending only on $\varepsilon$, such that
\[
c_{\varepsilon}g(x,\varepsilon)
\leq f(x,\varepsilon)
\leq C_{\varepsilon}g(x,\varepsilon).
\]
\end{definition}

In the proofs that follow, we will use the standard squared $L^2$ contraction estimate
\begin{equation}\label{eq:l2-contraction}
\mathbb{E}\left[\|v(t)-\bar v\|_2^2\right]
\leq
\exp\left(-\frac{t}{2\gamma(G)}\right)
\|v(0)-\bar v\|_2^2.
\end{equation}

\subsection{Counterexample construction and motivation}
\cite{movassagh2022repeatedaveragesgraphs} provides bounds on the order of $t_{\varepsilon,2\to1}$. We restate their bounds with the following theorem.

\begin{theorem}[adapted from Corollary 6 of \cite{movassagh2022repeatedaveragesgraphs}]\label{thm:general-bounds}
For every connected graph $G$ on $n\geq3$ vertices and every $\varepsilon\in(0,1)$,
\[
\gamma(G)\log(\varepsilon^{-1})
\leq
t_{\varepsilon,2\to1}(G)
\leq
4\gamma(G)\log(\sqrt{n}\,\varepsilon^{-1}).
\]
\end{theorem}

Recall that Conjecture 2 from \cite{movassagh2022repeatedaveragesgraphs} hypothesizes that
\[
t_{\varepsilon,2\to1}(G)
=
\Theta_{\varepsilon}\bigl(\gamma(G)\log n\bigr).
\]
Therefore, any counterexample to Conjecture 2 must exhibit faster mixing. Additionally, \cite{movassagh2022repeatedaveragesgraphs} describes a class of graphs for which the upper bound in Theorem~\ref{thm:general-bounds} is sharp.
In particular, they prove a theorem which shows, informally, that if a Fiedler eigenvector for a graph $G$ is \emph{delocalized}, in the sense required by their theorem, then
\[
t_{\varepsilon,2\to1}(G)
=
\Theta_{\varepsilon}\bigl(\gamma(G)\log n\bigr).
\]
We now define what it means for a vector to be delocalized, as in \cite{movassagh2022repeatedaveragesgraphs}.

\begin{definition}[\cite{movassagh2022repeatedaveragesgraphs}, Def.~6]
A vector $v\in\mathbb{R}^n$ is called $\delta$-delocalized if
\[
\|v\|_1\geq\delta\sqrt{n}\,\|v\|_2.
\]
\end{definition}

For many well-known families of graphs, such as the complete graph, star graph, and cycle graph, one can choose $\delta$-delocalized Fiedler eigenvectors as $n\to\infty$ \cite{movassagh2022repeatedaveragesgraphs}. When the relevant eigenvalue has multiplicity greater than one, such an eigenvector is not unique, so this statement concerns the indicated choice of eigenvector. Unsurprisingly, these graphs all have high degrees of symmetry, which may help produce delocalized choices. Hence, to motivate our counterexample, we considered alternate constructions involving asymmetric graphs.

Our idea was to join a small asymmetric component to a complete graph $K_n$ at a single vertex. Intuitively, the added component can reduce the spectral gap while the complete graph still supplies most of the vertices and edges. In the expression $\gamma(G)\log n$---the hypothesized order of the $L^2\to L^1$ mixing time---the factors can then reflect different parts of the graph. The simplest such construction is a single pendant vertex attached to the clique. In \cite{movassagh2022repeatedaveragesgraphs} it was shown that the spectral gap of $K_n$ is $n$; adding the pendant vertex reduces the spectral gap to $1$. We continue with a formal definition of our counterexample family and an analysis of its spectral gap.

Let $G_n$ be obtained by attaching a pendant vertex $\ell$ to a fixed vertex $A$ in $K_n$. Then
\[
|V|=n+1
\qquad\text{and}\qquad
|E|=\binom{n}{2}+1.
\]
Recall the following property of the spectral gap for any connected graph $G$ with Laplacian matrix $L$.

\begin{lemma}[Spectral gap as infimum]
\[
\lambda_2(G)
=
\inf_{w\perp\mathbf{1}}
\frac{w^\top Lw}{w^\top w}.
\]
\end{lemma}

This well-known characterization, called the Rayleigh quotient, follows from the symmetry of the Laplacian matrix. We next compute the spectrum of $G_n$ exactly.

\begin{lemma}[Laplacian spectrum of $G_n$]\label{lem:spectrum}
For $G_n$, a clique with a leaf, the Laplacian eigenvalues are
\[
0,\qquad
1,\qquad
n\text{ with multiplicity }n-2,\qquad
n+1.
\]
In particular,
\[
\lambda_2(G_n)=1
\]
and
\[
\gamma(G_n)
=
\frac{|E(G_n)|}{\lambda_2(G_n)}
=
\binom{n}{2}+1.
\]
\end{lemma}

\begin{proof}
Fix $n$. Let $L$ denote the Laplacian matrix of $G_n$, let $A$ denote the clique vertex adjacent to the leaf $\ell$, and let $B$ be the set of the other $n-1$ clique vertices.

Consider the $(n-2)$-dimensional subspace $U$ of vectors supported on $B$ whose coordinates sum to zero.
For $w\in U$ and $v\in B$,
\[
(Lw)_v
=
(n-1)w_v-\sum_{u\in B\setminus\{v\}}w_u
=
nw_v,
\]
while the coordinates at $A$ and $\ell$ are zero. Thus $Lw=nw$, so $n$ is an eigenvalue with multiplicity at least $n-2$.

It remains to consider the three-dimensional invariant subspace on which all vertices in $B$ have a common value $x$, the vertex $A$ has value $a$, and the leaf $\ell$ has value $r$. We compute that the eigenvalue equations are
\[
\lambda x=x-a,\qquad
\lambda a=na-(n-1)x-r,\qquad
\lambda r=r-a.
\]
Equivalently, the restriction of $L$ to this subspace is represented by
\[
\begin{pmatrix}
1&-1&0\\
-(n-1)&n&-1\\
0&-1&1
\end{pmatrix}.
\]
Its characteristic equation is
\[
\lambda(\lambda-1)(n+1-\lambda)=0,
\]
so the remaining eigenvalues are $0$, $1$, and $n+1$. Together with the eigenvalue $n$ on $U$, this accounts for all $n+1$ eigenvalues and proves the claim.
\end{proof}

\section{Proof of the main theorem}
We now prove Theorem~\ref{thm:main}. This theorem disproves Conjecture 2 from \cite{movassagh2022repeatedaveragesgraphs}, and further shows that the lower bound in Theorem~\ref{thm:general-bounds} is sharp.

Fix $\varepsilon\in(0,1)$ and an initial state with $\|v(0)\|_2\leq1$. Observe that
\[
\|v(0)-\bar v\|_2\leq1,
\]
where $\bar v$ is the average vector of $v$. Write $u(t)$ for the restriction of $v(t)$ to the $n$ clique vertices and let $y(t)$ be their average; let $r(t)$ be the leaf value and $z$ the global mean. Observe that $\bar v$ is the vector whose entries all equal $z$.

Let $c$ and $d$ be constants which may depend on the fixed value of $\varepsilon\in(0,1)$ but do not depend on $n$. Since the averaging process is discrete, it can only be run for integral time durations.
We will show that at time
\[
t_2
=
\left\lceil c\gamma(G_n)\right\rceil
+
\left\lceil dn\log n\right\rceil,
\]
we have
\[
\mathbb{E}\bigl[\|v(t_2)-\bar v\|_1\bigr]
\leq\varepsilon.
\]
We separate the averaging process up to time $t_2$ into two phases.

\medskip
\noindent\textbf{Phase 1 (length $t_1=\lceil c\gamma(G_n)\rceil$).}
By the standard squared $L^2$ contraction estimate~\eqref{eq:l2-contraction}, first shown by Aldous and Lanoue \cite{aldous2012lecture},
\[
\mathbb{E}\bigl[\|v(t_1)-\bar v\|_2^2\bigr]
\leq
e^{-t_1/(2\gamma(G_n))}
\|v(0)-\bar v\|_2^2
\leq
e^{-c/2}.
\]
The second inequality follows from $\|v(0)-\bar v\|_2\leq1$ and $t_1\geq c\gamma(G_n)$. Recall that $r(t)$ represents the leaf value at time $t$. By Cauchy--Schwarz,
\[
\begin{aligned}
\mathbb{E}\bigl[|r(t_1)-z|\bigr]
&\leq
\left(\mathbb{E}\bigl[|r(t_1)-z|^2\bigr]\right)^{1/2}\\
&\leq
\left(\mathbb{E}\bigl[\|v(t_1)-\bar v\|_2^2\bigr]\right)^{1/2}\\
&\leq e^{-c/4}.
\end{aligned}
\]
Next, we bound $\mathbb{E}[|y(t_1)-z|]$. By the definitions of $z$, $y$, and $r$, conservation of mass gives
\[
z(n+1)=ny(t_1)+r(t_1).
\]
Rearranging,
\[
n\bigl(z-y(t_1)\bigr)=r(t_1)-z.
\]
Taking absolute values and expectations,
\[
\mathbb{E}\bigl[|z-y(t_1)|\bigr]
=
\frac{1}{n}\mathbb{E}\bigl[|r(t_1)-z|\bigr]
\leq
\frac{1}{n}e^{-c/4}.
\]
Therefore, we have proven the following lemma.

\begin{lemma}[Distances after Phase 1]\label{lem:phase1}
We have
\[
\mathbb{E}\bigl[|r(t_1)-z|\bigr]\leq e^{-c/4}
\]
and
\[
\mathbb{E}\bigl[|y(t_1)-z|\bigr]
\leq
\frac{1}{n}e^{-c/4}.
\]
\end{lemma}

\medskip
\noindent\textbf{Phase 2 (length $\Delta t=\lceil dn\log n\rceil$).}
Recall that $t_2=t_1+\Delta t$, and let $E_s$ denote the edge selected at step $s$. Recall that $K_n$ is the $n$-clique subgraph of $G_n$. Define
\[
A
=
\{E_s\in E(K_n)\text{ for every }s=t_1+1,\ldots,t_2\}.
\]
On $A$, the clique average remains constant, so
\[
y(t)=y(t_1)
\qquad
(t_1\leq t\leq t_2),
\]
and the leaf value remains equal to $r(t_1)$. Since the original edge choices are independent and uniform on $E(G_n)$, conditional on $A$ the edges
\[
E_{t_1+1},\ldots,E_{t_2}
\]
are independent and uniformly distributed on $E(K_n)$. Thus, conditional on $A$, the clique evolves exactly as the averaging process on $K_n$.

By the union bound and $|E(G_n)|=\binom{n}{2}+1$,
\[
\mathbb{P}(A^c)
\leq
\frac{\Delta t}{|E(G_n)|}
\leq
\frac{dn\log n+1}{\binom{n}{2}+1}
\leq
\frac{4d\log n}{n}
\]
for all sufficiently large $n$.

The reference \cite{movassagh2022repeatedaveragesgraphs} gives
\[
\lambda_2(K_n)=n
\qquad\text{and hence}\qquad
\gamma(K_n)=\frac{n-1}{2}.
\]
Conditional on $A$, the contraction estimate~\eqref{eq:l2-contraction} gives
\[
\begin{aligned}
\mathbb{E}\bigl[
\|u(t_2)-y(t_1)\mathbf{1}\|_2^2
\mid A
\bigr]
&\leq
e^{-\Delta t/(2\gamma(K_n))}
\mathbb{E}\bigl[
\|u(t_1)-y(t_1)\mathbf{1}\|_2^2
\mid A
\bigr]\\
&\leq n^{-d}.
\end{aligned}
\]
Here the last inequality uses $2\gamma(K_n)=n-1$, $\Delta t\geq dn\log n$, and the bound
\[
\|u(t_1)-y(t_1)\mathbf{1}\|_2\leq1,
\]
justified next.

Indeed,
\[
\begin{aligned}
\|u(t_1)-y(t_1)\mathbf{1}\|_2^2
&=
\min_{a\in\mathbb{R}}
\|u(t_1)-a\mathbf{1}\|_2^2\\
&\leq
\|u(t_1)-z\mathbf{1}\|_2^2\\
&\leq
\|v(t_1)-\bar v\|_2^2\\
&\leq1.
\end{aligned}
\]

We have kept the time $t_1$ for $y$ because conditioning on $A$ implies that $y(t_1)=y(t_2)$. Then
\[
\begin{aligned}
\mathbb{E}\bigl[
\|u(t_2)-y(t_1)\mathbf{1}\|_1
\mid A
\bigr]
&\leq
\sqrt{n}\,
\mathbb{E}\bigl[
\|u(t_2)-y(t_1)\mathbf{1}\|_2
\mid A
\bigr]\\
&\leq
\sqrt{n}
\left(
\mathbb{E}\bigl[
\|u(t_2)-y(t_1)\mathbf{1}\|_2^2
\mid A
\bigr]
\right)^{1/2}\\
&\leq
n^{(1-d)/2}.
\end{aligned}
\]

Because $A$ depends only on edge choices after time $t_1$, it is independent of the state at time $t_1$. Therefore, Lemma~\ref{lem:phase1} gives the same bounds conditional on $A$. By the triangle inequality,
\[
\begin{aligned}
\mathbb{E}\bigl[
\|u(t_2)-z\mathbf{1}\|_1
\mid A
\bigr]
&\leq
n^{(1-d)/2}
+
n\,\mathbb{E}\bigl[
|y(t_1)-z|
\mid A
\bigr]\\
&\leq
n^{(1-d)/2}+e^{-c/4}.
\end{aligned}
\]
Using both bounds in Lemma~\ref{lem:phase1} and the identity $r(t_2)=r(t_1)$ on $A$, we obtain
\[
\mathbb{E}\bigl[
\|v(t_2)-\bar v\|_1
\mid A
\bigr]
\leq
n^{(1-d)/2}+2e^{-c/4}.
\]

If we remove the conditioning, then
\[
\begin{aligned}
\mathbb{E}\bigl[\|v(t_2)-\bar v\|_1\bigr]
={}&
\mathbb{E}\bigl[
\|v(t_2)-\bar v\|_1
\mid A
\bigr]\mathbb{P}(A)\\
&+
\mathbb{E}\bigl[
\|v(t_2)-\bar v\|_1
\mid A^c
\bigr]\mathbb{P}(A^c),
\end{aligned}
\]
and hence
\[
\begin{aligned}
\mathbb{E}\bigl[\|v(t_2)-\bar v\|_1\bigr]
\leq{}&
\mathbb{E}\bigl[
\|v(t_2)-\bar v\|_1
\mid A
\bigr]\\
&+
\mathbb{E}\bigl[
\|v(t_2)-\bar v\|_1
\mid A^c
\bigr]\mathbb{P}(A^c).
\end{aligned}
\]

Recall that
\[
\mathbb{P}(A^c)\leq\frac{4d\log n}{n}.
\]
Each averaging update is an $L^2$ contraction, so deterministically
\[
\begin{aligned}
\|v(t_2)-\bar v\|_1
&\leq
\sqrt{n+1}\,\|v(t_2)-\bar v\|_2\\
&\leq
\sqrt{n+1}\,\|v(0)-\bar v\|_2\\
&\leq
\sqrt{n+1}.
\end{aligned}
\]
Thus,
\[
\mathbb{E}\bigl[
\|v(t_2)-\bar v\|_1
\mid A^c
\bigr]
\leq
\sqrt{n+1}.
\]
Substituting,
\[
\mathbb{E}\bigl[\|v(t_2)-\bar v\|_1\bigr]
\leq
n^{(1-d)/2}
+
2e^{-c/4}
+
\frac{4d\log n}{n}\sqrt{n+1}.
\]

First choose $d>1$. Next choose $c=c(\varepsilon)$ sufficiently large that
\[
2e^{-c/4}\leq\frac{\varepsilon}{3}.
\]
Finally, take $n$ sufficiently large that
\[
n^{(1-d)/2}\leq\frac{\varepsilon}{3}
\]
and
\[
\frac{4d\log n}{n}\sqrt{n+1}
\leq
\frac{\varepsilon}{3}.
\]
Then the right-hand side of the preceding inequality is at most $\varepsilon$. This implies
\[
t_{\varepsilon,2\to1}(G_n)
\leq
t_2
=
\left\lceil c\gamma(G_n)\right\rceil
+
\left\lceil dn\log n\right\rceil
\]
by definition.

\medskip
\noindent\textbf{Summary.}
By Lemma~\ref{lem:spectrum},
\[
\lambda_2(G_n)=1
\]
and
\[
\gamma(G_n)
=
\binom{n}{2}+1
=
\Theta(n^2).
\]
Consequently,
\[
\begin{aligned}
t_{\varepsilon,2\to1}(G_n)
&\leq
t_2\\
&=
\left\lceil c\gamma(G_n)\right\rceil
+
\left\lceil dn\log n\right\rceil\\
&=
O_{\varepsilon}\bigl(\gamma(G_n)\bigr),
\end{aligned}
\]
because $n\log n=o(\gamma(G_n))$ and $c,d$ do not depend on $n$. The lower bound in Theorem~\ref{thm:general-bounds} gives
\[
t_{\varepsilon,2\to1}(G_n)
=
\Omega_{\varepsilon}\bigl(\gamma(G_n)\bigr).
\]
Therefore,
\[
t_{\varepsilon,2\to1}(G_n)
=
\Theta_{\varepsilon}\bigl(\gamma(G_n)\bigr)
=
\Theta_{\varepsilon}(n^2),
\]
completing the proof of Theorem~\ref{thm:main}. \qed

\section{Conclusion}
We have shown that Conjecture~2 of \cite{movassagh2022repeatedaveragesgraphs} does not hold in general by presenting the simple counterexample of a complete graph with a pendant vertex. Our two-phase argument demonstrates that the $L^2\!\to L^1$ mixing time is
$\Theta_{\varepsilon}(\gamma(G_n))$, rather than $\Theta_{\varepsilon}(\gamma(G_n)\log n)$ as conjectured.
Our counterexample has a strongly localized Fiedler eigenvector whose normalized squared $L^2$ mass is concentrated on the leaf, while the balancing mass is spread across the clique; this intuitively eliminates the extra $\log n$ factor. Our results highlight the need for revised upper bounds on the averaging process that might account for eigenvector localization or other graph properties, and suggest new directions for understanding how graph structure shapes consensus dynamics.

\newpage

\section*{Acknowledgments}
This research was part of a project which was conducted during the summer of 2025 as part of Caltech’s SURF program under the supervision of Lingfu Zhang, in the Division of Physics, Mathematics, and Astronomy. The project was partially funded by the SURF program and NSF grant DMS-2505625. 
\section*{Artificial Intelligence Usage}
All proofs were developed without AI assistance. OpenAI’s ChatGPT was used only to proofread the manuscript and polish its exposition.

\end{document}